\documentclass[reqno]{amsart}
\usepackage{mathrsfs}
\usepackage{amsmath}
\usepackage{amsthm}
\usepackage{amsfonts}
\usepackage{amssymb}
\usepackage{url}
\usepackage{enumerate}
\usepackage[pdftex,bookmarks=true]{hyperref}
  \usepackage[usenames, dvipsnames]{color}
\usepackage{verbatim}
\usepackage{graphicx}

\usepackage{pdfsync}

\newcommand\R{{\mathbf{R}}}
\newcommand\C{{\mathbf{C}}}

\renewcommand\P{{\mathbb{P}}}
\newcommand\E{{\mathbb{E}}}

\newcommand\T{{\mathbf{T}}}

\newcommand\Var{{\operatorname{Var}}}

\newcommand\Z{{\mathbf{Z}}}

\newcommand\eps{\varepsilon}

\newcommand\Bg{{\mathbf g}}

\newcommand\Bu{{\mathbf u}}
\newcommand\Bv{{\mathbf v}}

\newcommand\Bx{{\mathbf x}}

\newcommand\BS{{\mathbf S}}
\newcommand\BT{{\mathbf T}}

\newcommand\CA{{\mathcal A}}

\newcommand\CC{{\mathcal C}}

\newcommand\CE{{\mathcal E}}

\newcommand\CS{{\mathcal S}}

\newcommand\CU{{\mathcal U}}

\newcommand\CZ{{\mathcal Z}}

\renewcommand\mod{\ \operatorname{mod}\ }

\newcommand\ep{\varepsilon}

\newcommand{\al}{\alpha}
\newcommand{\la}{\lambda}
\newcommand{\Ent}{\operatorname{Ent}}

\usepackage{fullpage}

\theoremstyle{plain}
  \newtheorem{theorem}[subsection]{Theorem}
  \newtheorem{conjecture}[subsection]{Conjecture}

    \newtheorem{proposition}[subsection]{Proposition}
  
  \newtheorem{lemma}[subsection]{Lemma}
  \newtheorem{corollary}[subsection]{Corollary}
  \newtheorem{cor}[subsection]{Corollary}
 \newtheorem{question}[subsection]{Question}
  
  \newtheorem{condition}{Condition}

  \newtheorem{claim}[subsection]{Claim}

\theoremstyle{definition}

\begin{document}

\title[Random eigenfunctions on flat tori]{Concentration of the number of intersections of random eigenfunctions on flat tori}
\author{ Hoi H. Nguyen}

\address{Department of Mathematics, The Ohio State University, Columbus, Ohio 43210}
\email{nguyen.1261@math.osu.edu}


\thanks{The author is partially supported by National Science Foundation grant DMS-1752345.}

\maketitle
\begin{abstract}
We show that in
two dimensional flat tori the number of intersections between
random eigenfunctions of general eigenvalues and a given smooth curve is almost exponentially concentrated around its mean, even when the randomness is not gaussian. 
\end{abstract}

\section{Introduction}
Let $\T^2$ be the two dimensional flat tori $\R^2/\Z^2$. Let $F$ be a real-valued eigenfunction of the Laplacian on $\T^2$ with eigenvalue $\lambda^2$,
$$-\Delta F = \lambda^2 F.$$

It is known that all eigenvalues $\lambda^2$ have the form $4\pi^2 m$ where $m=a^{2}+b^{2}$ for some $a, b\in \Z$. Let $\CE_\lambda$ be the collection of $\mu=(\mu_1,\mu_2) \in \Z^2$ such that
$$\mu_1^2 + \mu_2^2 = m.$$
Denote $N=\# \CE_\lambda$. Note that if we express $m$ in the form $m=m_1^2 m_2$ with $m_1=2^r \prod_{q_k \equiv 3\mod 4} q_k^{b_k}$ and $m_2= 2^c \prod_{p_j \equiv 1\mod 4} p_j^{a_j}$ ($c=0,1$) then
$$N = \prod_j (a_j+1).$$
 
The toral eigenfunctions $F(x) =e^{2\pi i \langle \mu, x\rangle}, \mu \in \CE_\lambda$ form an orthonormal basis in the eigenspace corresponding to $\lambda^2$. For a given toral eigenfunction $F$ the nodal set $N_F$ is defined to be the zero set of $F$,
$$N_F :=\left \{x\in \T^2: F(x)=0\right \}.$$
The nodal set $N_F$ has been studied intensively in analysis and differential geometry. In this note we will be focusing on the intersection between $N_F$ and a given reference curve $\CC\subset \T^2$ parametrized by $\gamma: [0,1] \to \T^2$ with the following properties.

\begin{condition}[Assumption on $\gamma$] \label{cond-curve}  $\CC$ has unit length and $\gamma(t)$ is real analytic with positive curvature. More specifically, there exists a positive constant $c$
such that 
$$\|\gamma'(t) \|_2=1\mbox{ and } \|\gamma''(t)\|_2 >c \mbox{ for all } t.$$ 
\end{condition}

The number of nodal intersections $\CZ(F)$ between $F$ and $\CC$ is defined to be the cardinality of the intersection $N_{F}\cap \CC$.
$$\CZ(F) := \# \{x : x\in \CC \wedge F(x) = 0\}.$$

\subsection{Deterministic results} About ten years ago Bourgain and Rudnick provided uniform upper and lower bounds for the $L^{2}$-norm of the restriction of $F$ to $\CC$ as follows.
\begin{theorem}\cite[Main Theorem, Theorem 1.1]{BR1} \label{thm:restriction:size} Assume that $\CC$ is as in Condition \ref{cond-curve}. We have
\begin{equation}\label{eqn:restriction}
\int_\CC |F|^2 d \gamma =\Omega\left (\int_{\T^2} |F(x)|^2 dx\right ).
\end{equation}
Also, for any $\ep>0$,
$$\lambda^{1-\ep} \ll \CZ(F) \ll \lambda,$$
where the implicit constants depend only on $\CC$ and $\ep$ but not on $\lambda$.  
\end{theorem}
Here we say that $f = O(g)$, or $g=\Omega(f)$, or $f\ll g$, if there exists a constant $C$ such that $|f|\le C|g|$.

It was then conjectured by Bourgain and Rudnick that the lower bound is of order $\lambda$.
\begin{conjecture}\cite{BR1}\label{conj:BR} We have
$$\CZ(F) \gg \lambda .$$
\end{conjecture}
In a subsequent paper, to support this conjecture they showed
\begin{theorem}\cite[Theorem 1.1]{BR-Lp}\label{theorem:BR} Assume that $\CC$ is as in Condition \ref{cond-curve}, then
$$\CZ(F) \gg \frac{\lambda}{B_\lambda^{5/2}}$$
where $B_\lambda$ denote the maximal number of lattice points which lie on an arc of size $\sqrt{\lambda}$ on the circle $\|x\|_2=\lambda$,
$$B_\lambda := \max_{\|x\|_2 =\lambda} \# \left \{\mu \in \CE_{\lambda}: \|x-\mu\|_2 \le \sqrt{\lambda}\right \}.$$
\end{theorem}
In particular, as one can show that $B_\lambda\ll \log \lambda$ (see \cite{BR-Lp}), we have
$$\CZ(F)\gg \lambda/\log ^{5/2}\lambda.$$
The link in Theorem \ref{theorem:BR} between $\CZ(F)$ and $B_{\lambda}$ yields another interesting relationship between Bourgain-Rudnick conjecture \ref{conj:BR} and Cilleruelo-Granville conjecture \cite{CG} which predicts that $B_\lambda = O(1)$ uniformly. This is known to hold for almost all $\lambda^2$, see for instance \cite[Lemma 5]{BR}. 
It's worth noting that when the curvature of $\CC$ is zero, it could happen that $\lim\inf_{\lambda} \CZ(F) = 0$ (see for instance the construction in \cite{BR-Lp}.) 
 
For later use, we cite here one of the key technical ingredients in the proof of Theorem \ref{theorem:BR}.
\begin{theorem}\cite[Lemma 4.1]{BR-Lp}\label{thm:restriction} For each $\mu\in \CE$ let  $h_\mu(t) \in C^1[0,1]$ and $\eps_\mu \in \C$ with $\sum_\mu |\eps_\mu|^2 =1$. Let 
$$H(t) = \sum_\mu \eps_\mu h_\mu (t) e^{i \langle \mu, \gamma(t)\rangle}.$$ 
Then there exists a constant $C_0$ depending on $\CC$ such that 
$$\int_{0}^1 |H(t)|^2dt \le 2 \max_\mu \int_{0}^1 |h_\mu(t)|^2  dt+ C_0 \frac{N}{\la^{1/6}}\Big(\max_\mu \max_t |h_\mu(t)|^2 + \max_\mu \max_t |h_\mu(t)|  \max_\mu \int_0^1 |h_\mu'(t)|\Big).$$
\end{theorem}

We also refer the reader to \cite{B,BW} for related results for deterministic eigenfunctions, which were obtained by passing to randomized ones.

\subsection{Arithmetic random wave model} Recall that $N=N_m = \#\CE_{\lambda}$ is the dimension of the eigenspace corresponding to the eigenvalue $\lambda^{2}$. A probabilistic approach to the study of $\CZ(F)$ was introduced in the pioneer paper of Rudnick and Wigman \cite{RW}. Consider the random Gaussian eigenfunction
\begin{equation}
  F(x) =\frac{1}{\sqrt N} \sum_{\mu\in \CE_\lambda} \eps_\mu e^{2\pi i \langle \mu, x\rangle },\label{Gaussian_model}
 \end{equation} 
for all $x\in \T^{2}$, where $\eps_\mu$ are iid complex standard Gaussian  with a saving
 $$\eps_{-\mu} = \bar{\eps}_\mu.$$
 This saving ensures that $F$ is real-valued. The random function $F$ is called {\it arithmetic random wave} \cite{Berry}, which is a stationary Gaussian field because the correlation $\E (F(x) F(y))$ is invariant under translation. As we can also see, the law of this model is independent of the choice of the orthonormal basis of the eigenspaces.
 
Rudnick and Wigman showed that for all eigenvalues, {\it almost all} eigenfunctions satisfy Conjecture \ref{conj:BR}. More specifically, they showed the following.

\begin{theorem} \cite[Theorems 1.1, 1.2]{RW} \label{theorem:gaussian}
Let $\CC \subset \T^2$ be a smooth curve on the torus, with nowhere vanishing curvature and of total length one. Then
\begin{enumerate}
\item The expected number of nodal intersections is precisely
$$\E_{\Bg} \CZ(F) = \sqrt{2m}.$$
\item The variance is bounded from above as follows
$$\Var_{\Bg}(\CZ(F)) \ll \frac{m}{N}.$$
\item Furthermore, let $\{m\}$ be a sequence such that $N_m\to \infty$ and the Fourier coefficient $\left \{\widehat{\tau_m}(4)\right \}$ do not accumulate at $\pm 1$, then
$$\Var_\Bg(\CZ(F)) =  \frac{m}{N}\int_\CC \int_\CC \left( \sum_{\mu \in \CE_\lambda}4 \frac{1}{N}  \left \langle \frac{\mu}{|\mu|}, \dot{\gamma}(t_1) \right \rangle^2  \left \langle  \frac{\mu}{|\mu|}, \dot{\gamma}(t_2) \right \rangle^2 -1\right )dt_1dt_2 + O\left (\frac{m}{N^{3/2}}\right ).$$
\end{enumerate}
\end{theorem}

Here the subscript $\Bg$ is used to emphasize standard Gaussian randomness, and $\tau_m$ is the probability measure on the unit circle $S^1 \subset \R^2$ associated with $\CE_\lambda$,
$$\tau_{m} := \frac{1}{N} \sum_{\mu\in \CE_\lambda} \delta_{\mu/\sqrt{m}}.$$ 

We also refer the reader to \cite[Proposition 2.2]{RWY} by Rudnick et.al. where general estimates were given when the condition on $\left \{\widehat{\tau_m}(4)\right \}$ is lifted, and to \cite[Theorem 1.3]{RoW} by Rossi and Wigman for further extension when the first term in $\Var_\Bg(\CZ(F))$ vanishes.

\subsection{Our main results} The magnitude $m/N$ of the variance in Theorem \ref{theorem:gaussian} suggests that $\CZ(F)$ is concentrated around its mean. Indeed, by Markov's bound, for any $\eps>0$ we have that
\begin{equation}\label{eqn:Markov}
\P_{\Bg}(|\CZ(F) - \E \CZ(F)| \ge  \eps \la)\ll \frac{1}{N \eps^2}.
\end{equation}
Furthermore, the aforementioned work \cite{RoW} showed that the fluctuation of $\CZ(F)$ satisfies Central Limit Theorem. Perhaps it is natural to ask
\begin{question} How well is $\CZ_{\Bg}(F)$ concentrated around its mean?
\end{question}
As far as we are concerned, despite of significant breakthroughs regarding the statistics of $\CZ_{\Bg}(F)$ mentioned above, there has been no attempt to study this simple question. Relatedly, there has been a few results in the literature to study concentration for various models, including \cite{BDFZ,NS, NS-complex,NgZ, Rozen}, but unfortunately none of those works seem to be applicable here. With this note we hope to provide a robust method for these types of questions. In the first step we show
\begin{theorem}[Concentration of the gaussian case]\label{thm:main:g} Assume that $\gamma$ satisfies Conditions \ref{cond-curve}. Then there exist constants $c,c'$ such that for $N^{-c'}\le \eps \le c'/\log N$ we have 
$$\P(  |\CZ_{\Bg}(F) - \E \CZ_{\Bg}(F)| \ge \eps \la ) \le e^{ - c \eps^9 N}.$$
\end{theorem}

For $\eps \asymp 1$ we trivially have $\P(  |\CZ(F) - \E \CZ(F)| \ge \eps \la ) \le e^{ - c N/(\log N)^9}$, however here we conjecture that the logarithmic power can be removed.

We will show furthermore that $\CZ(F)$ is very well concentrated even for non-gaussian distributions. Here 
\begin{equation}\label{eqn:F}F(x) = \frac{1}{\sqrt{N}}\sum_{\mu\in \CE_\lambda} \eps_\mu e^{2\pi i \langle \mu,
  x\rangle},
\end{equation}
where $\eps_\mu =\eps_{1,\mu} + i \eps_{2,\mu}$ and
$\eps_{1,\mu},\eps_{2,\mu}, \mu \in \CE_\lambda$ are iid copies of a common random variable $\xi$ of mean zero and variance one,  and $\eps_{-\mu}
= \bar{\eps}_{\mu}$. 
We will denote by $\P_{\ep_\mu}, \E_{\ep_\mu}$, and $\Var_{\ep_\mu}$ the probability, expectation, and variance with respect to the random variables $(\ep_\mu)_{\mu\in \CE_{\lambda}}$.
 
\begin{theorem}[Concentration of the non-gaussian case]\label{thm:main} Let $C_0$ be a given positive constant, and suppose that either $1/C_0 < |\xi|<C_0$ with probability one, or 
that $\xi$ is continuous with density bounded by $C_0$ and satisfies the logarithmic Sobolev inequality with parameter $C_0$ in \eqref{eqn:logSobolev}. Assume that $\gamma$ satisfies Condition \ref{cond-curve}. Then for almost all $m$ there exist constants $c,c'$ such that for $N^{-c'}\le \eps \le c'/\log N$ we have 
$$\P(  |\CZ(F) - \E \CZ(F)| \ge \eps \la ) \le e^{ - c \eps^9 N}.$$
Furthermore, the above is true for all $m$ when $\xi$ is continuous.
\end{theorem}

We notice that in the Bernoulli case ($\xi$ takes value $\pm 1$ with probability 1/2) one cannot obtain anything better than $\exp(-\Theta(N))$. The main technical reason preventing us from covering for all $m$ is that in general we cannot rely on Theorem \ref{theorem:gaussian}. We will use Theorem \ref{theorem:general} instead, which in turn is known only for almost all $m$ for general ensembles.

We remark that Theorem \ref{thm:main}  can also be extended for almost all $m$ to other types of $\xi$ not necessarily bounded nor satisfying the logarithmic Sobolev inequality. For instance our result also covers the following cases.

\begin{itemize}
\item When  $|\xi|> 1/C_0$ with probability one and $|\xi|$ has sub-exponential tail. Then our method, by taking $C_0=N^{\delta'}$ in Theorem \ref{thm:bounded} with an appropriate $\delta'$, yields a sub-exponential concentration of type $\P(  |\CZ(F) - \E \CZ(F) | \ge \eps N ) =O(e^{-(\eps N)^{\delta}})$ for some constant $0<\delta<1$. 
 \vskip .1in
 \item Additionally, by the same argument, when $|\xi|> 1/C_0$ with probability one for given $C_0>0$ and when $\E (|\xi|^{C'})<\infty$ for some sufficiently large $C'$, then $\P(  |\CZ(F) - \E \CZ(F) | \ge \eps n ) = O((\eps N)^{-C})$ as long as $N^{-c'}\le \eps \le 1/\log N$.
 \end{itemize}

Finally, our result can be seen as a continuation of \cite{NgZ} where exponential concentration of the number of real roots of random trigonometric polynomials was shown. Although our general approach is similar to that of \cite{NgZ}, the technical details are very different. More specifically we have to incorporate various non-trivial results such as Theorem \ref{thm:restriction:size}, Theorem \ref{thm:restriction}, Theorem \ref{thm:largesieve}, Theorem \ref{theorem:general},  Theorem \ref{thm:repulsion}, Proposition \ref{prop:manyroots} for the current model, all seem to be of their own interest.


{\bf Notations.} We consider $\lambda$ as
an asymptotic parameter going to infinity and allow all other quantities to depend on $\lambda$ unless they are
explicitly declared to be fixed or constant.  As mentioned earlier, we write $X =
O(Y)$, $Y=\Omega(X)$, $X \ll Y$, or $Y \gg X$ if $|X| \leq CY$
for some fixed $C$; this $C$ can depend on
other fixed quantities such as the the parameter $C_0$ in the condition of $\xi$ and the curve $\gamma$.  If $X\ll Y$ and $Y\ll X$, we say that $Y = \Theta(X)$. 

Throughout the note, if not specified otherwise, a property $p(m)$ holds for {\it almost all} $m$ if the set of $m$ up to $T$ that $p(m)$ does not hold has cardinality much smaller than that of the set of $m$ for which $p(m)$ holds, i.e. $|\{m \le T, \bar{p}(m)\}| = o(|\{m \le T, p(m)\}|)$ as $T \to \infty$. Finally, all the norms $\|.\|_2$ in this note, if not specified, will be the usual $L_2$-norm.

\section{Supporting lemmas and proof method}
For $t\in [0,1]$, consider
\begin{equation}
 F(t)= \sum_\mu \eps_\mu  e^{i \langle \mu, \gamma(t)\rangle} =  \sum_\mu \eps_\mu  e^{i \la \langle \mu/\la, \gamma(t)\rangle}
 \end{equation}
with $\eps_\mu = \bar{\eps}_{-\mu}$.

For each positive integer $d=1,2$ let 
$$H_d(t) = \frac{\partial^d}{\partial^d t} F(t)=  \sum_\mu \eps_\mu  \la^d h_{d,\mu} (t) e^{i \la \langle \mu/\la, \gamma(t)\rangle},$$ 

\begin{theorem}[Restricted large sieve inequality]\label{thm:largesieve} Assume that $F(t)$ and $H_d(t)$ are as above, where $d=1,2$. Then for any $0\le x_1 <x_2 <\dots <x_M\le 1$, with $\delta$ being the minimum of the gaps between $x_i, x_{i+1}$, we have
$$\sum_{i=1}^M |H_d(x_i)|^2 \le C_1^{d} \la^{2d}  (\la+\delta^{-1}),$$
where $C_1$ depends on $\gamma$. 
\end{theorem}
\begin{proof}(of Theorem \ref{thm:largesieve}) It suffices to assume that $\delta\le x_1$ and $x_M \le 1-\delta$. We follow the classical approach by Gallagher \cite{Ga} with the important input of Theorem \ref{thm:restriction}.
\begin{claim} Let $g$ be a differentiable function on $I=[a-h,a+h]$. Then 
$$g(a) \le \frac{1}{|I|}\int_I |g(t)|dt + \frac{1}{2}\int_I |g'(t)|dt.$$
\end{claim}
\begin{proof} Let $\rho(t) = t-(a-h)$ if $t\in (a-h,a)$ and $\rho(t) = t-(a+h)$ if $t\in (a,a+h)$. Partal intergrals (over $(a-h,a)$ and $(a,a+h)$) give
$$\int_I \rho(t) g'(t) dt =2h g(a) - \int_{I} g(t)dt.$$
Note that $|\rho(.)|\le h$, so the claim follows by triangle inequality.
\end{proof}
By this claim, 
$$\sum_{i=1}^M |H_d(x_i)|^2 \le \frac{1}{\delta} \sum_i \int_{x_i-\delta/2}^{x_i+\delta/2} |H_d(t)|^2 dt + \sum_i \int_{x_i-\delta/2}^{x_i+\delta/2} |F(t) H_d'(t)| dt \le  \frac{1}{\delta} \int_{0}^1 |H_d(t)|^2 dt + \int_{0}^1 |H_d(t) H_d'(t)| dt.$$
Note that by Cauchy-Schwarz, $\int_{0}^1 |H_d(t) H_d'(t)| dt \le \sqrt{ \int_{0}^1 |H_d(t)|^2 dt } \sqrt{ \int_{0}^1 |H_d'(t)|^2 dt}$. The claim then follows from the  $L_2$-bound from Theorem \ref{thm:restriction}.
\end{proof}

On the probability side, for bounded random variables we will rely on the following consequence of McDiarmid's inequality.

\begin{theorem}\label{thm:bounded} Assume that $\Bx=(x_1,\dots, x_n)$, where $x_i$ are iid copies of $\xi$ of mean zero, variance one, and $|\xi| \le C_0$ with probability one. Let $\CA$ be a set in $\R^n$. Then for any $t>0$ we have
$$\P(\Bx \in \CA) \P(d_2(\Bx, \CA) \ge t \sqrt{n}) \le 4\exp(-t^4 n/16C_0^4).$$
\end{theorem}

For random variables $\xi$ satisfying the  log-Sobolev inequality, that is so that there is a positive constant $C_0$ such that for any smooth, bounded, compactly supported functions $f$ we have 
\begin{equation}\label{eqn:logSobolev}
\Ent_\xi(f^2) \le C_0 \E_\xi|\nabla f|^2,
\end{equation}
where $\Ent_\xi(f) = \E_\xi(f \log f) - (\E_\xi(f)) (\log \E_\xi(f))$, we  use the following.

\begin{theorem}\label{thm:sobolev}  Assume that $\Bx=(x_1,\dots, x_n)$, where $x_i$ are iid copies of $\xi$ satisfying \eqref{eqn:logSobolev} with a given $C_0$. Let $\CA$ be a set in $\R^n$. Then for any $t>0$ we have
$$\P\big(d_2(\Bx,A) \ge t\sqrt{n} \big) \le 2 \exp\big(-\P^2(\Bx\in \CA) t^2 n/4C_0 \big).$$
In particularly, if $\P(\Bx\in \CA)\ge 1/2$ then $\P(d_2(\Bx,A) \ge t\sqrt{n}) \le 2 \exp(-t^2 n/16C_0)$. Similarly if $\P(d_2(\Bx,A) \ge t\sqrt{n})\ge 1/2$ then $ \P(\Bx\in \CA)\le 2 \exp(-t^2 n/16C_0)$.
\end{theorem}
These results are standard, whose proof can be found for instance in \cite[Appendix B]{NgZ}.

Now we discuss the proof method for Theorem \ref{thm:main}. Broadly speaking, the approach follows a perturbation framework (see also \cite{NS,NgZ,Rozen} for recent adaptions) with detailed steps as follows:
\begin{enumerate}
\item Our starting point is an input from \cite{CNNV} which shows that $\E\CZ(F)$ is close to $\E_\Bg \CZ(F)$ and $\CZ(F)$ is moderately concentrated around its mean. 
\vskip .1in
\item We then show that it is highly unlikely that there is a small set of intervals where both $|F|$ and $|F'|$ are both small. We justify this by relying on a strong {\it repulsion estimate} (Theorem \ref{thm:repulsion}) and on a variant of {\it large sieve inequality} (Theorem \ref{thm:largesieve}). This step is carried out in Section \ref{section:exceptional}.
\vskip .1in
\item Furthermore, we will show in Section \ref{section:lowertail} via Jensen's bound that the number of roots over these intervals (where $|F|$ and $|F'|$ are small simultaneously) is small. 
\vskip .1in
\item Basing on the above results, we have that if $\CZ(F)$ is close to $\E \CZ(F)$ (with high probability, from the first step), then $\CZ(F+g)$ is also close to $\CZ(F)$ as long as $\|g\|_2$ is small. As such, geometric tools such as Theorem \ref{thm:bounded} and \ref{thm:sobolev} can be invoked to show that indeed $\CZ(F)$ satisfies exponential concentration. 
\end{enumerate}

\section{proof of Theorem \ref{thm:main}: preparation}
We first introduce a recent result \cite[Theorem 1.10, Theorem 1.13]{CNNV} which shows that the moments of $\CZ(F)$ are asymptotically universal.

\begin{theorem}[Universality of moment statistics]\label{theorem:general} Assume that $\gamma$ satisfies Condition \ref{cond-curve} and $\xi$ is as in Theorem \ref{thm:main}. Then for almost all $m$ we have

\begin{itemize}
\item $\E_{\ep_\mu} \CZ(F) = \E_{\Bg} \CZ(F) + O\left (\lambda/N^{c'}\right )$;
\vskip .05in
\item More generally, for any fixed $k$, $\E_{\ep_\mu} \CZ(F)^k = \E_{\Bg} \CZ(F)^k + O\left (\lambda^k/N^{c'}\right )$,
\end{itemize}
where $c'$ and the implicit constants in Conditions \ref{cond-curve} and $C_0$ but not on $N$ and $\lambda$. Furthermore, if $\xi$ is continuous and have bounded density function, then the above holds for all $m$. In particular, we have
$$\E_{\ep_\mu} \CZ(F) =  \sqrt{2m} + O\left (\lambda/N^{c'}\right ) \quad\mbox{ and } \quad\Var_{\ep_\mu}(\CZ(F)) \ll  \frac{\lambda^2}{N^{c'}}.$$
\end{theorem}

One crucial corollary of this result is that $\CZ(F)$ is already concentrated around its mean via Markov's bound
\begin{equation}\label{eqn:Markov'}
\P_{\eps_\mu}(|\CZ(F) - \E \CZ(F)| \ge  \eps \la)\ll \frac{1}{N^{c'} \eps^2}.
\end{equation}
This will serve as the starting point of our analysis.

Our next key ingredient, Theorem \ref{thm:repulsion} below, is a {\it repulsion-type} estimate which shows that at any point it is unlikely that the function and its derivative vanish simultaneously. 

First, recall that $t\in [0,\la]$ and
\begin{equation}
 H(t)= F(t/\la) = \frac{1}{\sqrt{N}}\sum_{\mu\in \CE_\lambda} \eps_{1,\mu} \cos (2\pi \langle \mu, \gamma(t/\la)\rangle) +  \eps_{2,\mu} \sin (2\pi \langle \mu, \gamma(t/\la)\rangle).
 \end{equation}
 and
 \begin{equation}
 H'(t) = \frac{1}{\sqrt{N}}\sum_{\mu\in \CE_\lambda} -\eps_{1,\mu} 2\pi \langle \mu/\la, \gamma'(t/\la)\rangle  \sin (2\pi \langle \mu, \gamma(t/\la)\rangle) +  \eps_{2,\mu} 2\pi \langle \mu/\la, \gamma'(t/\la)\rangle \cos (2\pi \langle \mu, \gamma(t/\la)\rangle).
 \end{equation}
 We prove our repulsion result via the study of small ball probability of the random walk $\frac{1}{\sqrt{N}} \sum_\mu \eps_{1,\mu}  \Bu_\mu + \eps_{2,\mu} \Bv_\mu$ where 
$$\Bu_\mu= \Big ( \cos (2\pi \langle \mu, \gamma(t/\la)\rangle), -2\pi \langle \mu/\la, \gamma'(t/\la)\rangle  \sin (2\pi \langle \mu, \gamma(t/\la)\rangle)\Big)$$
and
$$\Bv_\mu= \Big( \sin (2\pi \langle \mu, \gamma(t/\la)\rangle), 2\pi \langle \mu/\la, \gamma'(t/\la)\rangle  \cos (2\pi \langle \mu, \gamma(t/\la)\rangle)\Big).$$
We first show that these vectors are asymptotically isotropic.
\begin{claim}\label{claim:iso} For all $(a,b)\in \BS^1$ we have 
$$\sum_\mu \langle \Bu_\mu, (a,b) \rangle^2 +  \langle \Bv_\mu, (a,b) \rangle^2\asymp N.$$
\end{claim}
\begin{proof} We have
\begin{align*}
\sum_\mu \langle \Bu_\mu, (a,b) \rangle^2 +  \langle \Bv_\mu, (a,b) \rangle^2 &=  \sum_{\mu} \left[a \cos (2\pi \langle \mu, \gamma(t/\la)\rangle)- b \langle \mu/\la, \gamma'(t/\la)\rangle  \sin (2\pi \langle \mu, \gamma(t/\la)\rangle)\right]^2\\
&+ \left [a \sin (2\pi \langle \mu, \gamma(t/\la)\rangle) + b \langle \mu/\la, \gamma'(t/\la)\rangle  \cos (2\pi \langle \mu, \gamma(t/\la)\rangle)\right]^2 \asymp N\\
&= N a^2 + b^2 \sum_\mu (\langle \mu/\la, \gamma'(t/\la)\rangle)^2 \asymp N,
\end{align*}
where we used the fact $\|\gamma'(.)\|_2=1$ and if $\mu=(\mu_1,\mu_2)\in \CE_\la$ then $(\pm \mu_1, \pm \mu_2) \in \CE_\la$.
\end{proof}
Notice that $\|\Bu_\mu\|_2, \|\Bv_\mu\|_2 \ll 1$. The above claim implies that a positive portion of the $\{ |\langle \Bu_\mu, (a,b) \rangle|,  |\langle \Bv_\mu, (a,b) \rangle|\}$ are of order 1. Using this information we obtain the following key bound.

\begin{lemma}\label{lemma:smallball} For any $r \ge 1/\sqrt{N}$ we have 
$$\sup_{a\in \R^2}\P\left (\frac{1}{\sqrt{N}} \sum_\mu \eps_{1,\mu}  \Bu_\mu + \eps_{2,\mu} \Bv_\mu \in B(a,r)\right )= O(r^2).$$
\end{lemma} 
\begin{proof} This is \cite[Theorem 1]{Halasz} where we cover a ball of radius $r$ by $Nr^2$ balls of radius $1/\sqrt{N}$. 
\end{proof}
We also refer the reader to   \cite{NgV-survey,RV-rec} for further developments of similar anti-concentration estimates. We deduce from Lemma \ref{lemma:smallball} the following corollary.
\begin{theorem}[Repulsion estimate]\label{thm:repulsion} Assume that $\xi$ has mean zero and variance one.
 Then as long as $\al> 1/\sqrt{N}$,  $\beta>1/\sqrt{N}$, for every $t\in [-\la,\la]$ we have
$$\P\big( |H(t)|\le \al \wedge |H'(t)| \le \beta \big) = O(\al \beta).$$
\end{theorem}
In application we just choose $\al,\beta$ to be at least $N^{-c}$ for some small constant $c$.

\section{Exceptional polynomials}\label{section:exceptional} This current section is motivated by the treatment in \cite[Section 4.2]{NS} and \cite[Section 4]{NgZ}. Let $C>4$ be a sufficiently large number and choose 
\begin{equation}\label{eqn:R}
R=C\log N.
\end{equation} 
Cover $[0,1]$ by $\frac{\la}{R}$ open interval $I_i$ of length (approximately) $R/\la$ each. Let $3 I_i$ be the interval of length $3R/\la$ having the same midpoint with $I_i$. Given some parameters $\al, \beta$, we call an interval $I_i$ {\it stable} for a function $f$ if there is no point in $x\in 3I_i$ such that $|f(x)|\le \al$ and $|f'(x)|\le \beta \la$. In other words, there is no $x\in 3I_i$ where $|f(x)|$ and $|f'(x)|/\al$ are both small. Let $\delta$ be another small parameter (so that $\delta R <1/4$), we call $f$ {\it exceptional} if the number of unstable intervals  is at least $\delta \la$. We call $f$ not exceptional otherwise. 

For convenience, for each $F(x) = \frac{1}{\sqrt{N}}\sum_{\mu\in \CE_\lambda} \eps_{1,\mu} \cos (2\pi \langle \mu, \gamma(x)\rangle) +  \eps_{2,\mu} \sin (2\pi \langle \mu, \gamma(x)\rangle)$ we assign a unique (unscaled) vector $\Bv_{F}= (a_\mu, b_\mu)$ in $\R^{2N}$, which is a random vector when $F$ is random. 
Let $\CE_e = \CE_e(R,\al,\beta; \delta)$ denote the set of vectors $\Bv_{F}$ associated to exceptional function $F$. Our goal in this section is the following.
\begin{theorem}\label{thm:exceptional} Assume that $\al,\beta,\delta$ satisfy $\delta \le 1/4R$ and
\begin{equation}\label{eqn:parametersthm}
\al \asymp \delta^{3/2}, \beta \asymp \delta^{3/4},  \delta > N^{-2/5}.
\end{equation}
Then we have
$$\P\Big(\Bv_{F} \in \CE_e\Big) \le e^{-c \delta^8 N},$$
where $c$ is absolute.
\end{theorem}

We now discuss the proof. First assume that $f$ (playing the role of $F$) is exceptional, then there are $K=\lfloor \delta \la/3 \rfloor$ unstable intervals that are $R/\la$-separated (and hence $4/\la$-separated).
 Now for each unstable interval in this separated family we choose $x_j \in 3 I_j$ where  $|f(x_j)|\le \al$ and $|f'(x_j)|\le \beta n$ and consider the interval $B(x_j, \gamma/\la)$ for some $\gamma <1$ chosen sufficiently small (given $\delta$, see \eqref{eqn:parameters}). Let 
$$M_j:= \max_{x\in B(x_j,\gamma/\la)} |f''(x)|.$$  
By Theorem \ref{thm:largesieve} we have
$$\sum_{j=1}^K M_j^2 \le  \frac{2\la+(4/\la)^{-1}}{2\pi} \int_{x \in [0,1]} f''(x)^2 dx \le  \la^5 \frac{\sum_\mu \eps_{1,\mu}^2+ \eps_{2,\mu}^2}{N}.$$
On the other hand, in both the boundedness and the log-Sobolev cases we have $ \frac{\sum_\mu \eps_{1,\mu}^2+ \eps_{2,\mu}^2}{N}  \ge 4$ exponentially small, so without loss of generality it suffices to assume $\frac{\sum_\mu \eps_{1,\mu}^2+ \eps_{2,\mu}^2}{N} \le 4$. We thus infer from the above that the number of $j$ for which $M_j \ge C_2 \delta^{-1/2}\la^2$ is at most $2 C_2^{-2} \delta \la$. Hence for at least $(1/3 - 2 C_2^{-2})\delta \la$ indices $j$ we must have $M_j <C_2 \delta^{-1/2}\la^2$.

Consider our function over $B(x_j, \gamma/N)$, then by Taylor expansion of order two around $x_j$, we obtain for any $x$ in this interval
$$ |f(x)| \le \al + \beta \gamma + C_2 \delta^{-1/2} \gamma^2/2 \mbox{ and } |f'(x)| \le (\beta + C_2 \delta^{-1/2} \gamma) \la.$$
Now consider a function $g$ such that $\|g\|_2 \le \tau$. Our polynomial $g$ has the form $g(x)= \frac{1}{\sqrt{N}}\sum_{\mu\in \CE_\lambda} a'_{\mu} \cos (2\pi \langle \mu, \gamma(x)\rangle) +  b_{\mu}' \sin (2\pi \langle \mu, \gamma(x)\rangle)$, where $a_{\mu}', b_{\mu}'$ are the amount we want to perturb in $f$. Then as the intervals $B(x_j, \gamma/\la)$ are $4/\la$-separated, by Theorem \ref{thm:largesieve} we have
$$\sum_j \max_{x \in B(x_j, \gamma/\la)} g(x)^2 \le  8\la  \frac{\sum_\mu {a'}_{\mu}^2+ {b'}_{\mu}^2}{N} \le 8\la \tau^2$$
and 
$$\sum_j \max_{x \in B(x_j, \gamma/\la)} g'(x)^2 \le  8\la  \frac{\sum_\mu {a'}_{\mu}^2+ {b'}_{\mu}^2}{N} \le 8\la^3 \tau^2.$$
Hence, again by an averaging argument, the number of intervals where either $\max_{x \in B(x_j, \gamma/\la)} |g(x)| \ge C_3 \delta^{-1/2} \tau$ or   $\max_{x \in B(x_j, \gamma/\la)} |g'(x)| \ge C_3 \delta^{-1/2} \tau \la$ is bounded from above by $(1/3 - 2 C_2^{-2})\delta \la /2$ if $C_3$ is sufficiently large. On the remaining at least $(1/3 - 2 C_2^{-2})\delta \la /2$ intervals, with $h=f+g$, we have simultaneously that
$$|h(x)| \le  \al + \beta \gamma + C_2 \delta^{-1} \gamma^2/2 + C_3 \delta^{-1/2} \tau \mbox{ and } |h'(x)| \le  (\beta + C_2 \delta^{-1} \gamma + C_3 \delta^{-1/2} \tau) \la.$$
For short, let 
$$\al'= \al + \beta \gamma + C_2 \delta^{-1} \gamma^2/2 + C_3 \delta^{-1/2} \tau \mbox{ and } \beta'=\beta + C_2 \delta^{-1/2} \gamma + C_3 \delta^{-1/2} \tau.$$ 
It follows that $\Bv_h$ belongs to the set $\CU=\CU(\al, \beta,\gamma,\delta, \tau, C_1,C_2,C_3)$ in $\R^{2N}$ of the vectors  corresponding to $h$, for which the measure of $x$ with $|h(x)| \le  \al'$ and $|h'(x)| \le  \beta' \la $ is at least $(1/3 - 2 C_2^{-2})\delta \gamma$ (because this set of $x$ contains $(1/3 - 2 C_2^{-2})\delta \la /2$ intervals of length $2\gamma/\la$). 
Putting together we have obtained the following claim. 


\begin{claim}
\label{claim2} Assume that $\Bv_{f} \in \CE_e$. Then for any $g$ with $\|g\|_2 \le \tau$ we have $\Bv_{f+g} \in \CU$. In other words,
$$\Big \{\Bv\in \R^{2N}, d_2(\CE_e, \Bv) \le \tau \sqrt{N}\Big \} \subset \CU.$$
\end{claim}
We next show that $\P(\Bv_{f} \in \CU)$ is smaller than $1/2$. Indeed, for each $F$, let $B(f)$ be the measurable set of  $x\in \BT$  such that $\{|f(x)| \le \al'\} \wedge \{f'(x)| \le  \beta' \la \}$. Then the Lebesgue measure of $B(f)$, $\mu(B(f))$, is bounded by 
$$\E \mu(B(f)) = \int_{x \in  \BT} \P(\{|f(x)| \le \al'\} \wedge \{|f'(x)| \le \la \beta'\}) dx = O(\al' \beta'),$$
where we used Theorem \ref{thm:repulsion} for each $x$.  It thus follows that $\E \mu(B(f)) = O(\al' \beta')$. So by Markov inequality,
\begin{equation}
\label{eq-1011}
\P(\Bv_{f} \in \CU) \le \P\big ( \mu(B(f)) \ge (1/3 - 2 C_2^{-2})\delta \gamma \big ) = O(\al' \beta'/\delta \gamma) <1/2
\end{equation}
if $\al, \beta$ are as in \eqref{eqn:parametersthm} and then $\gamma, \tau$ are chosen appropriately, for instance as
\begin{equation}\label{eqn:parameters}
\gamma \asymp \delta^{5/4}, \tau \asymp \delta^2. 
\end{equation}


\begin{proof}(of Theorem \ref{thm:exceptional}) By Theorems \ref{thm:bounded} and \ref{thm:sobolev}, and by Claim \ref{claim2} and \eqref{eq-1011} we have
$$\P(\Bv\in \CE_e) \le e^{-c\tau^4 N}.$$
\end{proof}

\section{Roots over unstable intervals}\label{section:lowertail}

In this section we show the following lemma.

\begin{proposition}\label{prop:manyroots} Let $\eps$ be given as in Theorem \ref{thm:main}. Assume that the parameters $R, \al, \beta, \tau$ are chosen as in \eqref{eqn:R}, \eqref{eqn:parametersthm} and \eqref{eqn:parameters}. Assume that there are $\delta \la$ disjoint intervals $I$ of length $R/\la$ over which there are at least $\eps \la /2$ roots, then there exist a measurable set $A \subset [0,1]$ of measure at least $c\eps/4$ over which 
$$\max_{x\in A}|f(x)| \le \al \mbox{ and } \max_{x\in A} |f'(x)| \le \beta \la.$$ 
\end{proposition}

Before proving this result, we deduce that non-exceptional polynomials cannot have too many roots over the unstable intervals.

\begin{cor}\label{cor:manyroots}   Let the parameters $R, \eps, \al, \beta, \tau$ and $\delta$ be as in Proposition \ref{prop:manyroots}. Then a non-exceptional $F$ cannot have more than $\eps \la/2$ roots over any $\delta \la$ intervals $I_i$ from Section \ref{section:exceptional}. In particularly, $F$ cannot have more than $\eps \la/2$ roots over the unstable intervals. 
\end{cor}

\begin{proof}(of Corollary \ref{cor:manyroots}) If $F$ has more than $\eps \la/2$ roots over some $\delta \la$ intervals $I_i$, then  Proposition \ref{prop:manyroots} implies the existence of 
a set $A=A(F)$ that intersects with the set of stable intervals (because the total size of the unstable intervals is at most $\delta \la R/\la = \delta R <c\eps /8$), so that
 $\max_{x\in A}|F(x)| \le \al$ and $\max_{x\in A} |F'(x)| \le \beta \la$. However, this is impossible because for any 
 $x$ in the union of the stable intervals we have either $|F(x)|> \al$ or $|F'(x)| > \beta \la$.
\end{proof}


We now discuss the proof of Proposition \ref{prop:manyroots}. We first recall the following Jensen's bound, under analytic assumption
$$|\{w\in B(z,r): \psi(w)=0\}| \le \frac{\log\frac{M}{m}}{\log \frac{R^2 +r^2}{2Rr}}$$
where $M=\max_{w\in \bar{B}(z,R)}|\psi(w)|, m=\max_{w\in \bar{B}(z,r)}|\psi(w)|$.

In what follows 
$$F(z) := \sum_\mu a_\mu e^{i \langle \mu, \gamma(z) \rangle}.$$

\begin{lemma}\label{lemma:J}  Let $I$ be any interval with length $|I| \ll 1$. Assume that $\sum_\mu |a_\mu| \ll  N$. Then there exists a constant $c$ (depending on $\gamma$) such that
$$\left  |\left \{z \in I+B(0, \delta): F(z)=0\right \}\right  | \le c \lambda |I| +\log N - \log \max_{t\in I} \left  |F(t)\right  |.$$
\end{lemma}

\begin{proof}[Proof of Lemma \ref{lemma:J}] For $z\in I+B(0, 2|I| ), \exists t \in \R$ such that  $|z-t| < 4|I| $, 
$$|\gamma(z) - \gamma(t)| \le c |I|.$$
Hence for $\mu \in \CE_\lambda,$
$$\left |e^{i  \langle \mu, \gamma(z) \rangle }\right |= \left |e^{i  \langle \mu, \gamma(z)-\gamma(t)\rangle }  \right |\le e^{c \lambda |I|}.$$
Therefore 
$$\left  |F(z)\right  | \le \left (\sum_{\mu \in \CE_\lambda} |a_{\mu}|\right ) e^{c \lambda |I|} < \sqrt{N} e^{c \lambda |I|}.$$
Jensen's inequality (applied to the case that $B(z,4\delta), B(z,\delta)$) implies
\begin{align*}
\left  |\left \{z \in I+B(0, \delta), F(z)=0\right \}\right  | &\le \log \left (\sqrt{N} e^{c \lambda |I|}\right ) - \log \max_{t\in I} \left |F(t)\right  | \\
&\le c \lambda |I| + \log N -  \log \max_{t\in I}\left  |F(t)\right |. 
\end{align*}
\end{proof}

As a consequence we obtain the following

\begin{cor}\label{cor:J} Assume that $|I| \gg \frac{\log N}{\lambda}$. Assume that $\sum_\mu |a_\mu| \ll  N$ and one of the following holds, 
\begin{itemize}
\item $\max_{t\in I}|F(t)| \ge \exp(-c \la |I| /2)$; 
\vskip .1in
\item $\max_{t\in I}|F'(t)| \ge \la \exp(-c \la |I|  /2)$. 
\end{itemize}
Then we have 
$$ |\{t \in I, F(t)=0\}| \le 2 c |I|  \la.$$
\end{cor}
\begin{proof}(of Corollary \ref{cor:J}) It is clear that if $\max_{t\in I}|F(t)| \ge \exp(-c \la |I| /2)$ then Lemma \ref{lemma:J} implies the claim. Now assume $\max_{t\in I}|F'(t)| \ge \la \exp(-c \la |I|  /2)$. For $z\in I+B(0, 2|I| ), \exists t \in \R$ such that  $|z-t| < 4|I| $ such that $|\gamma(z) - \gamma(t)| \le c |I|$. Hence for $\mu \in \CE_\lambda$, as before we have $\left |e^{i  \langle \mu, \gamma(z) \rangle }\right | \le e^{c \lambda |I|}$ as well as $|\gamma'(z)| \le |\gamma'(t)|+ c'|I| = 1+c'|I|$. The later implies that $|\langle \mu/\la, \gamma'(z)\rangle| \ll 1$.  Therefore 
$$\left  |\frac{1}{\la}F'(z)\right  | \ll \left (\sum_{\mu \in \CE_\lambda} |a_{\mu}|\right ) e^{c \lambda |I|} \ll \sqrt{N} e^{c \lambda |I|}.$$
To this end, as $|\{t \in I, F(t)=0\}|  \le  |\{t \in I, F'(t)=0\}|+1$, and by Jensen's inequality the later can be bounded by 
\begin{align*}
\left  |\left \{z \in I+B(0, \delta), \frac{1}{\la}F'(z)=0\right \}\right  | &\le \log \left (\sqrt{N} e^{c \lambda |I|}\right ) - \log \max_{t\in I} \left |\frac{1}{\la}F'(t)\right  | \\
&\le c \lambda |I| + \log N +  c \lambda |I|/2  \le 2 c \lambda |I|. 
\end{align*}

\end{proof}

\begin{proof}(of Proposition \ref{prop:manyroots}) Among the $\delta \la$ intervals we first throw away those of less than $\eps \delta^{-1}/4$ roots, hence there are at least $\eps \la/4$ roots left. For convenience we denote the remaining intervals by $J_1,\dots, J_M$, where $M \le \delta \la$,  and let $m_1,\dots, m_M$ denote the number of roots over each of these intervals respectively. 


In the next step we  expand the intervals $J_j$ to larger intervals $\bar J_j$ 
(considered as union of consecutive closed intervals appearing at the beginning of Section \ref{section:exceptional}) of length $\lceil c m_j/R \rceil \times (R/\la)$ for some small constant $c$. Furthermore, if the expanded intervals $\bar J_{i_1}',\dots, \bar J_{i_k}'$ of $\bar J_{i_1},\dots, \bar J_{i_k}$ form an intersecting chain, then we create a longer interval
 $\bar J'$ of length $\lceil c(m_{i_1}+\dots+m_{i_k})/R \rceil \times (R/\la)$, which contains them and therefore contains at least $m_{i_1}+\dots+m_{i_k}$ roots. After the merging process,
 we obtain a collection 
 $\bar J_1',\dots, \bar J_{M'}'$ with the number of roots $m_1',\dots, m_{M'}'$ respectively, so that $\sum m_i'\geq \eps \la/2$.
  Note that now $\bar J_i'$ has length $\lceil cm_i'/R \rceil \times (R/\la) \approx cm_i'/\la$ (because $\eps \delta^{-1}$ is sufficiently large compared to $R$) and the intervals are $R/\la$-separated.
Now as $c$ is sufficiently small, over each $J_i'$ of length $cm_i'/\la \gg \log N/\la$ there are $m_i'$ roots, by Corollary \ref{cor:J} we must have
\begin{equation} 
\max_{t\in J_i'}|F(t)| \le \exp(-c m_i' /2) \mbox{ and } \max_{t\in J_i'}|F'(t)| \le \la \exp(-c m_i' /2).
\end{equation}
As $\al, \beta$ are of order at least $N^{-O(1)}$, so we automatically have in this case that $\max_{t\in J_i'}|F(t) \le \al$ and $\max_{t\in J_i'}|F'(t) \le \beta \la$.

Letting $A$ denote the union of all such intervals $J_{i}'$. Then we have  $\max_{x \in A} |F(x)| \le \al$ and $\max_{x \in A} |F'(x)|\le \beta \la$ and
$$\mu(A) \ge \sum_i cm_i'/\la \ge  c \eps/4.$$
\end{proof}


We conclude the section by a quick consequence of our lemma. For each $F$ that is not exceptional we let $S(F)$ be the collection of intervals over which $F$ is stable. Let $N_s(F)$ denote the number of roots of $F$ over the set $S(F)$ of stable intervals.

\begin{cor}\label{cor:control} With the same parameters as in  Corollary \ref{cor:manyroots}, we have
$$\P\Big(N_s(F) 1_{F \in \CE_e^c} \le \E \CZ(F)- \eps \la\Big)=o(1)$$
and
$$\E \Big(N_s(F) 1_{F \in \CE_e^c}\Big) \ge \E \CZ(F)- 2 \eps \la/3.$$
\end{cor}
\begin{proof}(of Corollary \ref{cor:control}) For the first bound, by Corollary \ref{cor:manyroots}, if $N_s(F) 1_{F \in \CE_e^c} \le \E \CZ(F)- \eps \la$ then $\CZ(F) 1_{F \in \CE_e^c} \le \E \CZ(F)- \eps \la/2$. Thus
\begin{align*}
\P\big(N_s(F) 1_{F \in \CE_e^c} \le \E \CZ(F)- \eps \la \big) &\le \P\big(\CZ(F)(F) 1_{F \in \CE_e^c} \le \E \CZ(F)- \eps \la/2 \big) \\
&\le \P\big(\CE_e^c \wedge \CZ(F) \le \E \CZ(F)- \eps \la/2 \big) + \P(\CE_e)=o(1),
\end{align*}
where we used \eqref{eqn:Markov} and Theorem \ref{thm:exceptional}. For the second bound regarding $\E(\CZ(F) 1_{F \in \CE_e^c})$ , let $N_{us}(F)$ denote the number of roots of $F$ over the set of unstable intervals. By Corollary \ref{cor:manyroots}, for non-exceptional $F$ 
 we have that
 $N_{us}(F)\le \eps \la/2$, and hence trivially $\E (N_{us}(F) 1_{F \in \CE_e^c}) \le \eps \la /2$. Because each $F$ has $O(\la)$ roots by Theorem \ref{thm:restriction:size}, we then obtain
\begin{align*}
\E (N_s(F) 1_{F \in \CE_e^c}) &\ge \E \CZ(F) - \E (N_{us}(F) 1_{F \in \CE_e^c}) - \E (\CZ(F) 1_{F \in \CE_e})\\
&\ge   \E \CZ(F) - \eps \la/2 - O( \la \times  e^{-c\tau^4 \la}) \ge \E \CZ(F) - 2\eps \la/3.
\end{align*}
\end{proof}

\section{proof of Theorem \ref{thm:main}: completion}


We first give a deterministic result (see also \cite[Claim 4.2]{NS}) to control the number of roots under perturbation.

\begin{lemma}\label{lemma:separation} Fix strictly 
positive numbers $\mu$ and $\nu$. Let $I=(a,b)$ be an  interval of length greater than $2\mu/\nu$, and let $f$ be a  $C^1$-function on $I$ such that at each point $x\in I$ we have either $|f(x)|> \mu$ or $|f'(x)| > \nu$. Then for each root $x_i \in I$ with $x_i-a>\mu/\nu$ and $b-x_i>\mu/\nu$ there exists an interval $I(x_i) = (a',b')$ where $f(a')f(b')<0$ and $|f(a')|=|f(b')|=\mu$,
such that $x_i \in I(x_i) \subset (x_i-\mu/\nu, x_i + \mu/\nu)$ and the intervals $I(x_i)$ over the roots are disjoint.
\end{lemma}

\begin{proof}(of Lemma \ref{lemma:separation}) We may and will assume that $f$ is not constant on $I$.
By changing $f(x)$ to $\la_1 f(\la_2 x)$ for appropriate $\lambda_1,\lambda_2$, it suffices to consider $\mu=\nu=1$.
For each root $x_i$, and for $0< t \le 1$ consider the interval $I_t(x_0)$ containing $x_0$ of those points $x$ where $|f(x)| < t$. We first show that for any $0<t_1,t_2\le 1$ we have that $I_{t_1}(x_1)$ and $I_{t_2}(x_2)$ are disjoint for distinct roots $x_i \in I$ satisfying the lemma's assumption. Assume otherwise, because $f(x_1)=f(x_2)=0$, there exists $x_1<x<x_2$ such that $f'(x)=0$ and $|f(x)| \le \min\{t_1,t_2\}$, and so contradicts with our assumption. We will also show that $I_1(x_0)\subset (x_0-1,x_0+1)$. Indeed, assume otherwise for instance that $x_0-1\in I_1(x_0)$, then for all $x_0-1<x<x_0$ we have $|f(x)|<1$, and so $|f'(x)|>1$ over this interval. Without loss of generality we assume $f'(x)>1$ for all $x$ over this interval. The mean value theorem would then imply that $|f(x_0-1)|= |f(x_0-1)-f(x_0)|>1$, a contradiction with $x_0-1 \in I_1(x_0)$.  As a consequence, we can define $I(x_i)=I_1(x_i)$, for which at the endpoints the function behaves as desired.  
\end{proof}

\begin{corollary}\label{cor:separation} Fix positive $\mu$ and $\nu$. Let $I=(a,b)$ be an  interval of length at least $2 \mu/\nu$, and let $f$ be a $C^1$-function on $I$  such that at each point $x\in I$ we have either $|f(x)|> \mu$ or $|f'(x)| > \nu$. Let $g$ be a function such that $|g(x)|< \mu$ over $I$. Then for each root $x_i \in I$ of $f$ with $x_i-a>\mu/\nu$ and $b-x_i>\mu/\nu$ we can find a root $x_i'$ of $f+g$ such that $x_i' \in (x_i-\mu/\nu, x_i + \mu/\nu)$, and also the $x_i'$ are distinct.
\end{corollary}

Now we prove Theorem \ref{thm:main} by considering the two tails separately.

\subsection{The lower tail}  We need to show that
\begin{equation}\label{eqn:lowertail}
\P(  \CZ(F)  \le \E \CZ(F)  - \eps \la ) \le e^{-c \eps^9 \la}.
\end{equation}

With the parameters $\al,\beta, \delta, \tau,R$ chosen as in  Corollary \ref{cor:manyroots},
 consider a non-exceptional polynomial $F$. Let $g$ be an eigenfucntion with $\|g\|_2 \le \tau$, where $\tau$ is chosen as in \eqref{eqn:parameters}.  Consider a stable interval $I_j$ with respect to $F$ (there are at least $(\frac{2\pi}{R} -\delta)\la$ such intervals). We first notice that the number of stable intervals $I_j$ over which $\max_{x \in 3I_j} |g(x)| > \al$ is at most at most $O(\delta \la)$. Indeed, assume that there are $M$ such intervals $3I_j$. Then we can choose $M/6$ such intervals that are $R/\la$-separated. By Theorem \ref{thm:largesieve} we have $(M/6) \al^2 \le \la \tau^2$, which implies $M \le 6\la (\tau \al^{-1})^2 =O(\delta \la)$. From now on we will focus on the stable intervals with respect to $F$ on which $|g|$ is smaller than $\al$. 


By Corollary \ref{cor:separation} (applied to $I= 3I_j$ with $\mu=\al$ and $\nu =\beta n$, note that $\al/ \beta \asymp \delta^{3/4} <R$), because $\max_{x \in 3I_j} |g(x)| < \al$, the number of roots of $F+g$  over each interval $I_j$ is at least as that of $F$. Hence if $F$ is such that $\CZ(F)\ge \E \CZ(F) - \eps \la/2$ and also $F$ has at least $ \E \CZ(F) - 2\eps \la/3$ roots over the stable intervals, then by Corollary \ref{cor:manyroots}, with appropriate choice of the parameters, $F$ has at least $ \E \CZ(F) - \eps \la$ roots over the stable intervals $I_j$ above where $|g| \le \al$, and hence Corollary \ref{cor:separation} implies that $F+g$ has at least $\E \CZ(F)-\eps \la$ roots over these stable intervals $I_j$. In particularly $F+g$ has at least $\E \CZ(F)-\eps \la$ roots over $\BT$. Let $\CU^{lower}$ be the collection of $\Bv_{F}$ from such $F$ (where $\CZ(F)\ge \E \CZ(F) - \eps \la/2$ and $F$ has at least $ \E \CZ(F) - 2\eps \la/3$ roots over the stable intervals).  Then by Corollary \ref{cor:control} 
and \eqref{eqn:Markov'} 
\begin{equation}\label{eqn:Uast}
\P(\Bv_{F} \in \CU^{lower}) \ge  1- \P\big(\CZ(F)\le \E \CZ(F) - \eps \la /2\big) - \P\big(N_s(F) 1_{F \in \CE_e^c} \le \E \CZ(F)- 2\eps \la/3 \big)  \ge 1/2.
\end{equation}

 \begin{proof}(of Equation \eqref{eqn:lowertail}) By our application of Corollary \ref{cor:separation} above, the set $\{\Bv, d_2(\Bv,\CU^{lower}) \le \tau\sqrt{2\la}\}$ is contained in the set of having at least $\E \CZ(F) - \eps \la$ roots. Furthermore, \eqref{eqn:Uast} says that $\P(\Bv_{F} \in \CU^{lower}) \ge 1/2$. Hence by Theorems \ref{thm:bounded} and \ref{thm:sobolev} 
$$
\P( \CZ(F)  \ge \E \CZ(F)  - \eps \la) \ge \P\Big(\Bv_{F}\in \big\{\Bv, d_2(\Bv,\CU^{lower}) \le \tau\sqrt{\la}\big \}\Big) \ge 1-\exp(- c \eps^9 \la),
$$
where we used the fact that $\tau \asymp \delta^2$ from \eqref{eqn:parameters}.
 \end{proof}

\subsection{The upper tail}
Our goal here is to justify the upper tail
\begin{equation}\label{eqn:uppertail}
\P(  \CZ(F)  \ge \E \CZ(F)  + \eps \la ) \le e^{-c \eps^9  \la}.
\end{equation}

Let $\CU^{upper}$ denote the set of $\Bv_{F}$ for which  $\CZ(F) \ge E \CZ(F)+ \eps \la$. By Theorem \ref{thm:exceptional} it suffices to assume that $F$ is non-exceptional.

 
 \begin{proof}(of Equation \eqref{eqn:uppertail}) Assume that for a non-exceptional $F$ we have $\CZ(F) \ge \E \CZ(F) + \eps \la$. Then by  Corollary \ref{cor:manyroots} the number of roots of $F$ over the stable intervals is at least $ \E \CZ(F) +2\eps \la/3$. Let us call the collection of $\Bv_{F}$ of these polynomials by $\CS^{upper}$. Then argue as in the previous subsection (with the same parameters of $\al, \beta, \tau, \delta$), Corollary \ref{cor:manyroots} and Corollary \ref{cor:separation} imply that any $h=F+g$ with $\|g\|_2 \le \tau$ has at least $\E \CZ(F) +  \eps \la /2$ roots. On the other hand, we know by \eqref{eqn:Markov} that the probability that $F$ belongs to this set of functions is smaller than $1/2$. It thus follows by Theorems \ref{thm:bounded} and \ref{thm:sobolev} that 
$$\P(\Bv_{F}\in  \CU^{upper}) \le e^{-c \eps^9  \la},$$
where we again used that $\tau \asymp \delta^{2}$.
\end{proof}

{\bf Acknowledgements.} The author is grateful to prof.~I.~Wigman for many helpful comments.


\end{document}